\documentclass[11pt,reqno]{amsart}
\usepackage{epsfig,color}
\usepackage{amsmath,amsfonts,amssymb,amstext,amsbsy,amsopn,amsthm}%,accents
\usepackage{esint}
\usepackage{graphicx}   % for figures
\usepackage{hyperref}

\numberwithin{equation}{section}
\theoremstyle{plain}
\newtheorem{theorem}[equation]{Theorem}

\newtheorem{proposition}[equation]{Proposition}
\newtheorem{lemma}[equation]{Lemma}
\newtheorem{corollary}[equation]{Corollary}

\theoremstyle{remark}
\newtheorem{remark}[equation]{Remark}

\theoremstyle{definition}
\newtheorem{definition}[equation]{Definition}

\newcommand{\Ric}{{\rm Ric}}
\newcommand{\Rm}{{\rm Rm}}

\newcommand{\Vol}{{\rm Vol}}
\newcommand{\vol}{{\rm vol}}

\newcommand{\dN}{\mathbb{N}}

\newcommand{\cI}{\mathcal{I}}

\newcommand{\cN}{\mathcal{N}}

\newcommand{\cV}{\mathcal{V}}
\newcommand{\cW}{\mathcal{W}}

\newcommand{\R}{\mathbb{R}}

\renewcommand{\epsilon}{\varepsilon}

\begin{document}
\title{New logarithmic Sobolev inequalities and \\an $\epsilon$-regularity theorem for the Ricci flow}
\author{Hans-Joachim Hein}
\address{Department of Mathematics, Imperial College, London SW7 2AZ, United Kingdom}
\email{h.hein@imperial.ac.uk}
\author{Aaron Naber}
\address{Department of Mathematics, MIT, Cambridge MA 02139, USA}
\email{anaber@math.mit.edu}
\thanks{HJH acknowledges postdoctoral support under EPSRC Leadership Fellowship EP/G007241/1. }
\date{\today}
\begin{abstract}
In this note we prove a new $\epsilon$-regularity theorem for the Ricci flow. Let $(M^n,g(t))$ with $t\in [-T,0]$ be a Ricci flow and $H_{x_0}$ the conjugate heat kernel centered at a point $(x_0,0)$ in the final time slice. Substituting $H_{x_0}$ into Perelman's $\cW$-functional produces a monotone function
$\cW_{x_0}(s)$ of $s \in [-T,0]$, the pointed entropy, with $\cW_{x_0}(s) \leq 0$, and $\cW_{x_0}(s) = 0$ iff $(M,g(t))$ is isometric to the trivial flow on $\R^n$. Our main theorem asserts
the following:  There exists an $\epsilon>0$,
depending only on $T$ and on lower scalar curvature and $\mu$-entropy bounds for $(M,g(-T))$, such that $\cW_{x_0}(s) \geq -\epsilon$ implies $|\Rm|\leq r^{-2}$ on $P_{\epsilon r}(x_0,0)$, where $r^2 = |s|$ and $P_\rho(x,t) \equiv B_\rho(x,t)\times (t-\rho^2,t]$.

The main technical challenge of the theorem is to prove an effective Lipschitz bound in $x$ for the $s$-average of $\cW_x(s)$.  To accomplish this, we require a new log-Sobolev inequality. It is well known by Perelman that the metric measure spaces $(M,g(t),d\vol_{g(t)})$ satisfy a log-Sobolev; however we prove that this is also true for the conjugate heat kernel weighted spaces $(M,g(t),H_{x_0}(-,t)\,d\vol_{g(t)})$.
Our log-Sobolev constants for these weighted spaces are in fact universal and sharp.

The weighted log-Sobolev has other consequences as well, including an average Gaussian
upper bound on the conjugate heat kernel that only depends on a two-sided scalar curvature bound.
\end{abstract}
\maketitle
\markboth{Hans-Joachim Hein and Aaron Naber}{New logarithmic Sobolev inequalities and an $\epsilon$-regularity theorem for the Ricci flow}
%\tableofcontents
\section{Introduction}\label{s:intro}
Throughout this paper we will assume that the pair
\begin{align}\label{e:RF}
(M^n, g(t)),\;\,t\in [-T,0],
\end{align}
is a smooth Ricci flow. For simplicity we assume each time slice is complete of bounded geometry. It may also be convenient to assume $M$ is compact when we rely on Perelman's monotonicity formula. However, our results will ultimately only depend on entropy and scalar curvature bounds.

Given a point $(x_0,0)\in M\times [-T,0]$ on the final time slice we write
\begin{equation}\label{e:conjhk}
H_{x_0}(y,s) = H(x_0,0\,|\,y,s) = (4\pi |s|)^{-\frac{n}{2}}\exp(-f_{x_0}(y,s))
\end{equation}
for the conjugate heat kernel based at $(x_0,0)$, and
\begin{equation}\label{e:conj_heat_meas}
d\nu_{x_0}(y,s) = H_{x_0}(y,s)\,d\vol_{g(s)}(y)
\end{equation}
for the associated probability measures on $M$; see also Definition \ref{d:conjugate_hk}.

\begin{definition}[Perelman \cite{P}]\label{d:entropy}
Let $(M^n,g)$ be a Riemannian manifold. Given an $f \in C^\infty(M)$ and
$\tau > 0$ such that $(4\pi\tau)^{-\frac{n}{2}}e^{-f}d\vol$ has unit mass, we define associated entropy functionals
\begin{align}\label{e:entropy}
 \cW(g,f,\tau) &\equiv \int [\tau(|\nabla f|^2 + R)+f-n](4\pi \tau)^{-\frac{n}{2}}e^{-f}d\vol,\\
\mu(g,\tau) &\equiv \inf \left\{\mathcal{W}(g,f,\tau) : \int (4\pi\tau)^{-\frac{n}{2}} e^{-f} d\vol = 1\right\}.
\end{align}
\end{definition}

Perelman \cite{P} discovered that $\cW$ is nondecreasing in $t = t_0 - \tau$ ($t_0 \in \R$) if $(4\pi\tau)^{-\frac{n}{2}}e^{-f}$ evolves by the conjugate heat equation coupled to the Ricci flow in the time variable $t < t_0$; see Theorem \ref{t:pfirstvar}. As a consequence, the quantity $\mu(g(t), t_0-t)$ is nondecreasing in $t  < t_0$ along any Ricci flow.

\subsection{Pointed entropy and $\epsilon$-regularity}\label{s:ptd_ent_eps_reg}

Our primary concern in this paper will be a localized version of the $\cW$-entropy. Namely, given a Ricci flow  as in (\ref{e:RF}), there exists for each point $(x_0,0)$ in the $0$-time slice and $s\in [-T,0)$ the canonical metric probability space $(M^n,g(s), d\nu_{x_0}(s))$,
where $d\nu_{x_0}(s)\equiv H_{x_0}(-,s)\,d\vol_{g(s)}$ is the conjugate heat kernel measure as in (\ref{e:conj_heat_meas}).

\begin{definition}\label{d:pointed_entropy} The pointed entropy at scale $\sqrt{|s|}$ based at $x_0$ is defined by
\begin{align}
 \cW_{x_0}(s)\equiv \cW(g(s),f_{x_0}(s),|s|).
\end{align}
\end{definition}

It is a consequence of Perelman's gradient formula, see Proposition \ref{p:ptd_entropy}, that
\begin{align}
\lim_{s\to 0}\cW_{x_0}(s) = 0, \;\,\frac{d}{ds}\cW_{x_0}(s) \geq 0.
\end{align}
Moreover, the pointed entropy at a given point and scale vanishes if and only if the flow is isometric to
 the trivial flow on Euclidean space.  One can view this as an example of a rigidity theorem.

What is often useful in such situations is an ``almost rigidity'' statement: If the pointed entropy {\it at a single point $x_0$} is close to $0$, then the Ricci flow ought to be smoothly close to flat $\R^n$ {\it near $x_0$}. There are many examples of such statements in the literature. For instance, in the case of Einstein manifolds one has an $\epsilon$-regularity theorem based on the volume ratio $\cV_x(r)\equiv \log(\Vol(B_r(x))/\omega_n r^n)$ going back to Anderson \cite{anderson}, which is at the backbone of the regularity theory in \cite{CN}.

To make all of this precise we introduce the \emph{regularity scale} of a point. Given $(x,t)\in M\times [-T,0]$ a bound on the curvature $|\Rm|(x,t)$ at this point gives remarkably little information. On the other hand, a bound on the curvature in a spacetime neighborhood of $(x,t)$ tells us everything about the local geometry of the Ricci flow.  Thus we make the following definition:

\begin{definition}
Given a Ricci flow $(M^n,g(t))$ as in (\ref{e:RF}) we define the following.
\begin{enumerate}
\item Given $(x,t)\in M\times [-T,0]$ and $r>0$ with $-T\leq t-r^2$ we define the parabolic ball
\begin{align}
P_r(x,t)\equiv B_r(x,t)\times (t-r^2,t].
\end{align}
\item Given $(x,t)\in M\times [-T,0]$ we define the regularity scale
\begin{align}
r_{|\Rm|}(x,t)\equiv \sup\{r>0:\sup_{P_r(x,t)}|\Rm|\leq r^{-2}\}.
\end{align}
\end{enumerate}
\end{definition}

\begin{remark}
The regularity scale is defined to be scale invariant: If $r_{|\Rm|}(x,0) = r$ and if we rescale the Ricci flow by $r^{-1}$, so that $P_r(x,0)$ is mapped to $P_1(\tilde x,0)$, then we have $|\Rm| \leq 1$ on $P_{1}(\tilde x,0)$.
\end{remark}
\begin{remark}
The regularity scale controls not only the curvature but also its derivatives.  Namely, by standard parabolic estimates there exists for each $k\in \dN$ a dimensional constant $C(n,k)$ such that if $r_{|\Rm|}(x,t)\equiv r$ then we have the estimates
\begin{align}
\sup_{P_{\frac{r}{2}}(x,t)}|\nabla^{k}\Rm|\leq C(n,k)r^{-2-k}.
\end{align}
\end{remark}

Our main theorem then takes the following form: \emph{There exists an $\epsilon>0$ such that $\cW_{x_0}(s)\geq -\epsilon$ implies $r_{|\Rm|}(x_0,0)^2\geq\epsilon |s|$.} Of course the use of such an estimate is only as good as what $\epsilon$ depends on, and we are able to bound $\epsilon$ below purely in terms of $T$ and lower scalar curvature and $\mu$-entropy bounds at the initial time $-T$, which is really the most one might hope for.

\begin{theorem}\label{t:eps_regularity}
For each $C > 0$ there exists an $\epsilon = \epsilon(n,C) > 0$ such that the following holds. Let $(M^n, g(t))$ be any Ricci flow as in \eqref{e:RF} such that
\begin{align}\label{e:basic_control}
R[g(s)]\geq-\frac{C}{|s|},\;\,\inf_{\tau \in (0,2|s|)}\mu(g(s),\tau)\geq -C,
\end{align}
for some $s \in [-T,0)$. If the pointed entropy satisfies
\begin{align}\label{e:entropy_small}
\cW_{x_0}(s) \geq -\epsilon
\end{align}
for some point $x_0$ in the $0$-time slice, then we have
\begin{align}\label{e:regularity_bound}
r_{|\Rm|}(x_0,0)^2 \geq\epsilon |s|.
\end{align}
\end{theorem}

\begin{remark}
One typically takes the point of view that $(M,g(-T))$ is a fixed Riemannian manifold but that the runtime $T$ of the flow is variable. One would then like to \emph{derive} (\ref{e:basic_control}) from information about the initial time slice $(M, g(-T))$.
The appropriate conditions to impose are
\begin{align}\label{e:basic_control_2}
R[g(-T)] \geq -\frac{C}{|s|},\;\,\inf_{\tau \in (T-|s|,T+|s|)}\mu(g(-T),\tau) \geq -C,
\end{align}
because both $\inf R[g(t)]$ and $\mu(g(t), t_0 - t)$  for any fixed $t_0 > t$ are nondecreasing in $t$. Even if one then only cares about fixed values of $s$, (\ref{e:basic_control_2}) still exhibits an explicit $T$-dependence in the entropy condition which may degenerate as $T \to\infty$. This issue is familiar from \cite{P}.
\end{remark}

The known $\epsilon$-regularity theorems for the Ricci flow either allow $\epsilon$ to depend on a type I sectional curvature bound \cite{EMT}, or require a smallness condition such as (\ref{e:entropy_small}) to hold \emph{at every $x$ in a definite neighborhood of $x_0$} \cite{Ni}, which clearly follows from (\ref{e:entropy_small}) under a type I condition.

In order to understand the technical difficulties encountered in the proof of Theorem \ref{t:eps_regularity}, let us give a brief outline of the argument. We would like to follow Anderson's proof \cite{anderson} of his $\epsilon$-regularity theorem for Einstein manifolds
as much as possible. Namely, assume the theorem fails. Then for all $\epsilon > 0$ we can find Ricci flows $(M_\epsilon,g_\epsilon(t))$ and points $(x_\epsilon,0)$ such that (\ref{e:basic_control}) and (\ref{e:entropy_small}) are satisfied but (\ref{e:regularity_bound}) fails. We can assume $s_\epsilon = 1$ by rescaling. After a careful point picking we would then like to say that $x_\epsilon$ more or less minimizes
the function $r_{|\Rm|}(x,0)$ in a small but definite ball \emph{while still satisfying $\cW_{x_\epsilon}(1) \geq -2\epsilon$}. Then,
after rescaling so that $r_{|\Rm|}(\tilde x_\epsilon,0)=1$, we can take a subsequence
converging smoothly to a complete pointed Ricci flow $(\tilde{M}_\infty,\tilde g_\infty(t), \tilde{x}_\infty)$ such that $r_{|\Rm|}(\tilde x_\infty,0)=1$, yet $\cW_{\tilde x_\infty}(s)=0$ for all $s$ and hence $(\tilde{M}_\infty,\tilde g_\infty(t)) \cong \R^n$, which is the desired contradiction.

However, in order for this to go through we require $\cW_x(s)$ to depend on $x$ in a Lipschitz manner (for the volume ratio $\cV_x(r)$ in the Einstein case this is clear by volume comparison). We are in fact unable to prove this assuming only (\ref{e:basic_control}). What does turn out to be possible under (\ref{e:basic_control}), though, is to obtain Lipschitz control on the \emph{$s$-average of $\cW_s(x)$}, which is still enough for our purposes. To accomplish this we introduce new log-Sobolev estimates for heat kernel measures. In addition, this $s$-average turns out to be an interesting monotone quantity in its own right; we will put this to some use in obtaining new integral bounds for the conjugate heat kernel, see Theorem \ref{t:hk_upper_grad_int}.

\subsection{Log-Sobolev inequalities for the conjugate heat kernel measure}
Consider a smooth metric probability space $(M,g,d\nu)$, where $d\nu = e^{-f}d\vol_g$. If the Bakry-{\'E}mery condition
\begin{align}\label{e:BE}
\Ric+\nabla^2f\geq \frac{1}{2} g
\end{align}
is satisfied, then a celebrated classical theorem from \cite{BE} asserts that $(M,g, d\nu)$ satisfies a log-Sobolev inequality.  That is, for every smooth function $\phi$ with compact support on $M$ we have
\begin{align}\label{e:log_sob}
\int\phi^2\, d\nu = 1 \;\, \Longrightarrow\;\,
\int \phi^2\log \phi^2 \,d\nu \leq 4\int |\nabla\phi|^2\,d\nu.
\end{align}
The case of flat $\R^n$ equipped with the Gaussian measure $d\nu = (4\pi)^{-\frac{n}{2}}\exp(-\frac{1}{4}|x|^2)dx$ is already very interesting; this case is equivalent to the original log-Sobolev inequality proved by Gross \cite{gross}.

The log-Sobolev (\ref{e:log_sob}) can be used to prove that the operator $\Delta_f \phi \equiv \Delta \phi - \langle \nabla f, \nabla \phi\rangle $, which is $d\nu$-selfadjoint, has discrete
spectrum. Linearizing (\ref{e:log_sob}) around $\phi \equiv 1$ yields a Poincar{\'e} for $d\nu$ that tells us that
the smallest positive eigenvalue of $\Delta_f$ is at least $\frac{1}{2}$, and it follows from \cite{BE} that equality holds if and only if $M$ splits off a line, in which case the $\frac{1}{2}$-eigenfunction is linear.

\begin{remark}
If $M$ is compact, an improvement depending on ${\rm diam}\,M$ for the spectral gap $\frac{1}{2}$  was recently proved in \cite{futaki}. Intriguingly, on every nontrivial normalized gradient shrinking Ricci soliton, $\Ric + \nabla^2 f = \frac{1}{2}g$, the soliton function $f$ itself is always a $1$-eigenfunction of $\Delta_f$.
\end{remark}

In the context of a Ricci flow $(M^n,g(t))$ with its conjugate heat kernel measures (\ref{e:conj_heat_meas}), Perelman's monotonicity formula (\ref{e:perelmono}) would suggest that the metric probability spaces $(M, g(s), d\nu_{x_0}(s))$ are typically well approximated by shrinking solitons. Thus the Bakry-{\'E}mery inequality (\ref{e:log_sob}) may at least help to motivate (if not prove) the following result, which is our main technical tool.

\begin{theorem}\label{t:ineqs}
Let $(M^n,g(t))$ be a Ricci flow as in \eqref{e:RF}. Fix $x_0 \in M$ and $s \in [-T,0)$.

{\rm (1)} For all $\phi \in C^\infty_0(M)$ with $\int \phi\,d\nu_{x_0}(s) = 0$,
\begin{equation}\label{e:poinc}
\int \phi^2 \,d\nu_{x_0}(s) \leq 2|s|\int |\nabla \phi|^2_{g(s)}\,d\nu_{x_0}(s).
\end{equation}
Equality holds if and only if either $\phi \equiv 0$, or $(M, g(t)) = ({M}', {g}'(t)) \times (\R, dz^2)$ isometrically for all $t \in [s,0]$ with $z(x_0) = 0$ and $\phi = \lambda z$ for some constant $\lambda \in \R^*$.

{\rm (2)} For all $\phi \in C^\infty_0(M)$ with $\int \phi^2\,d\nu_{x_0}(s) = 1$,
\begin{equation}\label{e:logsob}
\int \phi^2 \log \phi^2\,d\nu_{x_0}(s) \leq 4|s|\int |\nabla \phi|^2_{g(s)}\,d\nu_{x_0}(s).
\end{equation}
Equality holds if and only if either $\phi \equiv 1$, or $(M, g(t)) = ({M}', {g}'(t)) \times (\R, dz^2)$ isometrically for all $t \in [s,0]$ with $z(x_0) = 0$ and $\phi = \exp({\lambda z} - 2\lambda^2|s|)$ for some constant $\lambda \in \R^*$.
\end{theorem}

\begin{remark} Perelman's monotonicity formula implies an unweighted log-Sobolev inequality (\ref{e:rflogsob}) whose optimal constant depends on $T$ and on various bounds for the geometry of $(M, g(-T))$. The weighted inequalities in Theorem \ref{t:ineqs} on the other hand are sharp and completely universal. As far as we can tell, however, the applications of (\ref{e:rflogsob}) and (\ref{e:logsob}) are essentially disjoint.
\end{remark}

\begin{remark}
A static version of Theorem \ref{t:ineqs} for the heat kernel measure on complete Riemannian manifolds with ${\rm Ric} \geq 0$ was proved independently by Bakry-Ledoux \cite{BL} and (with a spurious extra factor of $n$) Bueler \cite{B}.
These papers inspired our proof of Theorem \ref{t:ineqs}.
\end{remark}

\subsection{Integral and pointwise bounds for the conjugate heat kernel}
Before returning to our discussion of the $\epsilon$-regularity result, Theorem \ref{t:eps_regularity}, we wish to explain an interesting consequence of Theorem \ref{t:ineqs}. The basic idea is to test (\ref{e:logsob}) with functions $\phi$ that are well adapted to the metric geometry; this is the so-called ``Herbst argument'' in metric measure theory \cite{L}.

The following appears to be the sharpest possible result such methods can yield.

\begin{theorem}\label{t:concentration}
Let $(M^n, g(t))$ be a Ricci flow as in \eqref{e:RF} and let $d\nu = d\nu_{x_0}(s)$ be a conjugate heat kernel measure as in \eqref{e:conj_heat_meas}. Then the Gaussian concentration inequality
\begin{equation}\label{e:concentration}
\nu(A)\nu(B) \leq \exp\left(-\frac{1}{8|s|}{\rm dist}_{g(s)}(A,B)^2\right)
\end{equation}
holds for all $A, B \subseteq M$. Here ${\rm dist}$ refers to the usual set distance, not the Hausdorff distance.
\end{theorem}

Let us observe the following natural consequence. Given $x_1,x_2\in M$ in the $s$-time slice, we can apply (\ref{e:concentration}) to $d\nu \equiv d\nu_{x_2}$, choosing $A,B$ to be the metric balls $B_r(x_1,s), B_r(x_2,s)$ with  $r^2 \equiv |s|$. If $B_r(x_1,s)$ is in addition noncollapsed with a uniform constant, which is of course a fair assumption after Perelman \cite{P}, then we obtain the following average Gaussian upper bound:
\begin{align}\label{e:weak_gaussian}
\fint_{B_r(x_1,s)} H_{x_2}(s)\,d\vol_{g(s)}\leq \frac{C|s|^{-\frac{n}{2}}}{\nu_{x_2}(B_r(x_2,s))} \exp\left(-\frac{1}{C|s|}d_{g(s)}(x_1,x_2)^2\right).
\end{align}
The primary concern with this estimate is a lack of effective lower bound on $\nu_{x_2}(B_r(x_2,s))$.
We can fix this to some extent by bringing in a pointwise Gaussian lower bound for $H_{x_2}$ from \cite{zhang2}, see Theorem \ref{t:hk_lower_bound}, however this forces us to work with time-$0$ balls, not time-$s$ balls.

\begin{corollary}\label{c:gaussian_integral_upper}
For each $C > 0$ there exists a $C' = C'(n,C) > 0$ such that the following holds. Let $(M^n,g(t))$ be any Ricci flow as in \eqref{e:RF} such that, for some $s \in [-T,0)$,
\begin{align}
\sup_{t \in [s,0]}\|R[g(t)]\|_\infty \leq \frac{C}{|s|},\;\,\inf_{\tau \in (0,2|s|)} \mu(g(s),\tau)\geq -C.
\end{align}
Let $x_1, x_2 \in M$ and put $r^2 \equiv |s|$. Then we have an average Gaussian upper bound
\begin{align}\label{e:uppergauss2}
\fint_{B_r(x_1,0)} H_{x_2}(s) \, d\vol_{g(s)}\leq C'|s|^{-\frac{n}{2}} \exp\left(-\frac{1}{C'|s|} {\rm dist}_{g(s)}(B_r(x_1,0), B_r(x_2,0))^2\right).
\end{align}
As a consequence, we get the following distance distortion type estimate:
\begin{align}\label{e:distance_distortion}
{\rm dist}_{g(s)}(B_r(x_1,0), B_r(x_2,0))&\leq C'd_{g(0)}(x_1,x_2).
\end{align}
\end{corollary}

\begin{remark}
Both (\ref{e:weak_gaussian}) and (\ref{e:uppergauss2}) fall short of what one might hope to have. For
example, there is a method of proving \emph{pointwise} Gaussian upper bounds for heat kernels on static manifolds relying on nothing more than a log-Sobolev
\cite{D}.
Unfortunately this approach seems to break down for the Ricci
flow for lack of control on the distance distortion between different time slices.
\end{remark}

\subsection{Lipschitz continuity of the pointed Nash entropy}\label{s:nash_cont}
We now return to the main flow of the argument and explain how the Poincar{\'e} inequality of Theorem \ref{t:ineqs}(1) helps us to complete the proof of Theorem \ref{t:eps_regularity}. As we said at the end of Section \ref{s:ptd_ent_eps_reg}, what we require is a Lipschitz
bound in $x$ for $\cW_x(s)$, or at least for a weaker quantity than $\cW_x(s)$ that still controls the soliton behavior of our Ricci flow near $x$. It turns out that we can work with the \emph{time average} of $\cW_x(s)$:

\begin{definition}\label{d:nash_entropy}
Given $x_0 \in M$ and $s \in [-T,0)$, we define the pointed Nash entropy by
\begin{align}\label{e:nashdef}
\cN_{x_0}(s) \equiv \frac{1}{|s|} \int_s^0 \cW_{x_0}(r)\,dr = \int_M f_{x_0}(s)\,d\nu_{x_0}(s) - \frac{n}{2}.
\end{align}
\end{definition}

See Proposition \ref{p:nash_entropy} for the equality in (\ref{e:nashdef}) and other basic properties. Thus, $\cN_{x_0}(s)$ is closely related to the quantity used by Nash \cite{nash} in proving H\"older continuity of weak solutions.

\begin{theorem}\label{t:ent_cont}
For each $C > 0$ there exists a $C' = C'(n,C) > 0$ such that the following holds. Let $(M^n,g(t))$ be any Ricci flow as in \eqref{e:RF} such that
\begin{align}\label{e:gen_assns_2}
R[g(s)]\geq -\frac{C}{|s|},\;\,\inf_{\tau \in (0,2|s|)} \mu(g(s),\tau)\geq -C,
\end{align}
for some $s \in [-T,0)$.  Then the mapping
\begin{align}
(M, g(0)) \ni x \mapsto f_x(s) H_x(s) \in L^1(M, d\vol_{g(s)})
\end{align}
is globally $C'|s|^{-\frac{1}{2}}$-Lipschitz. In particular this implies that
\begin{align}\label{e:nash_ent_cont}
|\cN_{x_1}(s) - \cN_{x_2}(s)| \leq C'|s|^{-\frac{1}{2}}d_{g(0)}(x_1,x_2).
\end{align}
\end{theorem}

The idea here, and the reason for introducing $\cN_x$ in the first place, is that $\cN_x$, unlike $\cW_x$, does not depend on derivatives of $H$. Thus, $\nabla_x\cN_x$ can be bounded by the gradient estimates of \cite{zhang_gradient, zhang1}.  The Poincar{\'e} (\ref{e:poinc}) comes in as a crucial tool to control the $L^2$-norm of $f_{x}(s)$ in the process.

Finally, with Theorem \ref{t:ent_cont} in hand it is not difficult to complete the proof of Theorem \ref{t:eps_regularity} (or in fact of a slightly strengthened version in which $\cW_{x_0}$ gets replaced by $\cN_{x_0}$) because the contradiction argument outlined in Section \ref{s:ptd_ent_eps_reg} works the same way with $\cN_{x_0}$ in place of $\cW_{x_0}$.

\section{Background material}\label{s:basics}

\subsection{The heat operator and its conjugate}

Let $(M^n, g(t))$ be a Ricci flow as in (\ref{e:RF}).

\begin{definition}
The heat operator and its conjugate along the flow are defined by
\begin{align}
\Box &\equiv \partial_t - \Delta,\\
\Box^* &\equiv -\partial_t - \Delta + R.
\end{align}
\end{definition}

This is a sensible definition because of the following identity.

\begin{lemma}\label{l:ibp}
Let $\Omega$ be a smooth bounded domain in $M$ and $[t_1,t_2] \subseteq [-T,0]$. Then
\begin{align}\label{e:ibp}
\int_{t_1}^{t_2} \int_{\Omega} (\Box u)v - (\Box^*v)u = \int_{\Omega} uv\; \biggr|_{t_1}^{t_2} - \int_{t_1}^{t_2} \int_{\partial\Omega} \left(\frac{\partial u}{\partial n}v - \frac{\partial v}{\partial n}u\right)
\end{align}
for all smooth functions $u, v: M \times [t_1,t_2] \to \R$.
\end{lemma}

\begin{definition}\label{d:conjugate_hk}
For $x,y \in M$ and $s < t$ in $[-T,0]$, we let $H(x,t\,|\,y,s)$ denote the conjugate heat kernel based at $(x,t)$, i.e.~the unique minimal positive solution to the equations
\begin{eqnarray}
\Box_{y,s}^*H(x,t\,|\,y,s) = (-\partial_s - \Delta_{y,\,g(s)} + R(y,s) )H(x,t\,|\,y,s) = 0,\\
\lim_{s \to t} H(x,t\,|\,y,s) = \delta_{x}(y).\end{eqnarray}
\end{definition}

If we wish to fix a point $x\in M$ we may write $H_x(y,s)\equiv H(x,0\,|\,y,s)\equiv (4\pi|s|)^{-\frac{n}{2}}e^{-f_x(y,s)}$.

\begin{lemma}\label{l:hk_properties}
The conjugate heat kernel satisfies the following properties.
\begin{enumerate}
\item $\int H(x,t\,|\,y,s)\,d\vol_{g(s)}(y) =1$.
\item $H(x,t\,|\,y,s)$ is also the fundamental solution of $\Box_{x,t} = \partial_t - \Delta_{x,\,g(t)}$ with pole at $(y,s)$.
\item $\int H(x,t\,|\,y,s) \, d\vol_{g(t)}(x) \leq \exp(\rho(t-s))$, where $\rho \equiv \|R[g(-T)]^-\|_\infty$.
\end{enumerate}
\end{lemma}

\begin{proof} (1) is simply the mass conserving property of the conjugate heat equation.

(2) requires some more care, and indeed reflects a very general fact concerning the fundamental solutions of parabolic operators and their formal adjoints. Let us write
\begin{align}
u &\equiv H_{\Box}(\,-,-\,|\,y,s),\\
v &\equiv H_{\Box^*}(x,t\,|\,-,-),
\end{align}
for the fundamental solutions of $\Box$ and $\Box^*$ with poles at $(y,s)$ and $(x,t)$ respectively. We then apply (\ref{e:ibp}) to $u$ and $v$ on the following three domains in the limit as $\epsilon \to 0$:
$$
[s,s+\epsilon^2] \times (M \setminus B_{\epsilon}(y,s)), \;\; [s+\epsilon^2, t-\epsilon^2] \times M,\;\; [t-\epsilon^2,t] \times (M \setminus B_{\epsilon}(x,t)).
$$
Using standard local asymptotics for $u$ and $v$ at their poles, this yields $u(x,t) = v(y,s)$.

(3) follows by differentiating $\int H(x,t\,|\,y,s) \, d\vol_{g(t)}(x)$ by $t$, substituting $\Box_{x,t}H = 0$ from (2), and using that the minimum of the scalar curvature is nondecreasing along the Ricci flow.\end{proof}

Finally, we recall a key tool that has already been put to good use in \cite{McCannTopping, zhang_gradient}.

\begin{lemma}\label{l:rf_bochner}
The following parabolic Bochner formula holds for all spacetime functions $u$:
\begin{align}\label{e:rf_bochner}
\Box \frac{1}{2}|\nabla u|^2 = -|\nabla^2 u|^2 + \langle\nabla u, \nabla\Box u\rangle.
\end{align}
\end{lemma}

\begin{proof}
Using $\frac{\partial}{\partial t}|\nabla u|^2 = 2\langle \nabla \frac{\partial}{\partial t} u, \nabla u\rangle + 2\Ric(\nabla u, \nabla u)$, this reduces to the usual computation. The Ricci term here cancels with the Ricci term from the standard elliptic Bochner formula.
\end{proof}

\subsection{Properties of the entropy functionals} We reviewed the definitions of Perelman's entropy functionals $\cW(g,f,\tau)$ and $\mu(g,\tau)$ in Section \ref{s:intro}. In addition we introduced localized versions $\cW_{x_0}(s)$ and $\cN_{x_0}(s)$. In this section we collect basic properties and applications of these functionals.

We begin with Perelman's foundational monotonicity formula and two corollaries \cite{KL,P}.

\begin{theorem}\label{t:pfirstvar}
Fix a smooth probability measure $dv$ on $M$ and let $f(t)$, $g(t)$ be families of functions and metrics on $M$ parametrized by $t \in [-T,0]$. Fix any $t_0 \in \R$ and put $\tau(t) \equiv t_0 - t$. If
\begin{equation}\label{e:gradient_flow}
\frac{\partial g}{\partial t} = -2(\Ric + \nabla^2 f), \;\; (4\pi\tau)^{-\frac{n}{2}} e^{-f} d\vol = dv,
\end{equation}
for all $t < t_0$, then we have
\begin{equation}\label{e:perelmono}
\frac{d}{dt}\mathcal{W}(g,f,\tau) = 2\tau \int \left|\Ric + \nabla^2 f - \frac{g}{2\tau}\right|^2 dv.
\end{equation}
\end{theorem}

Up to correcting by the diffeomorphisms generated by $\nabla f$, (\ref{e:gradient_flow}) is equivalent to $g$ evolving by Ricci flow and $(4\pi\tau)^{-\frac{n}{2}} e^{-f} $ evolving by the conjugate heat equation associated with this Ricci flow. For the rest of the present section, $(M^n, g(t))$ will denote a Ricci flow as in (\ref{e:RF}).

\begin{corollary}
For all $t_0 \in \R$, the quantity $\mu(g(t), t_0-t)$ is nondecreasing in $t < t_0$. The quantity is constant if and only if the flow is isometric to a gradient shrinking soliton with singular time $t_0$ and soliton function $f(t)$,
where $f(t)$ denotes any minimizer in the definition of $\mu(g(t), t_0-t)$.
\end{corollary}

\begin{corollary}\label{c:rflogsob}
For a fixed $t_0 \in \R$, define $\mu_0 \equiv \mu(g(-T), t_0+T)$ and write $\tau \equiv t_0 - t$. Then
\begin{equation}\label{e:rflogsob}
\int \phi^2 \log \phi^2 \, d\vol \leq
 \tau \int (4|\nabla\phi|^2 + R\phi^2)  \, d\vol - \frac{n}{2} \log 4\pi\tau - n - \mu_0
\end{equation}
holds for all $\phi \in C^\infty_0(M)$ with $\int \phi^2 \,d\vol = 1$ as long as $t < t_0$.
\end{corollary}

We will then also need to recall how (\ref{e:rflogsob}) implies Perelman's no local collapsing
theorem \cite{P}. We state an improved version of this result that only requires an upper scalar curvature bound \cite{KL}. As reported in \cite{KL}, this is due to Perelman as well.
The way we organize the proof may be slightly simpler than the version in \cite{KL} and was inspired by an argument in \cite{carron}; see also \cite{zhang1}.

\begin{theorem}\label{t:stronger_nlc}
Fix any $t \in [-T,0]$, $x \in M$, and $r > 0$, and suppose that we have
\begin{align}
\inf_{\rho \in (0,r)}\mu(g(-T), t + T + \rho^2) \geq -C, \;\, \sup_{B_r(x,t)} R[g(t)] \leq C r^{-2}.
\end{align}
Defining $\kappa \equiv \exp(-(2^{n+4} + 2C))$ it then follows that
\begin{align}|B_r(x,t)| \geq \kappa r^n.\end{align}
\end{theorem}

\begin{proof}
We work in the $t$-time slice. For $\rho \in (0,r]$ define a Lipschitz function $\psi$ to be $\equiv 1$ on $B(x,\frac{\rho}{2})$, $\equiv 0$ off $B(x,\rho)$, and linear in $d(x, -)$ in between, and apply (\ref{e:rflogsob}) to $\phi \equiv \psi/\|\psi\|_2$. Using Jensen's inequality with respect to $\frac{d\vol}{|B(x,\rho)|}$ on $B(x,\rho)$ to bound the left-hand side from below, we obtain
$$
\log \frac{1}{|B(x,\rho)|} \leq \frac{16\tau}{\rho^2}\left(\frac{|B(x,\rho)|}{|B(x,\frac{\rho}{2})|} - 1\right) + \frac{C\tau}{\rho^2} - \frac{n}{2}\log(4\pi \tau)- n - \mu_0.
$$
This holds for any given $\tau > 0$, with $\mu_0$ depending on $\tau$ by definition. We now make $\tau \equiv \rho^2$. Using our definition of $\kappa$, it is then easy to prove the following implication:
$$
\left|B\left(x,\frac{\rho}{2}\right)\right| \geq \kappa\left(\frac{\rho}{2}\right)^n \;\Longrightarrow\;  |B(x,\rho)| \geq \kappa \rho^n.
$$
Since $\kappa$ is smaller than the volume of $B_{\R^n}(1)$, the claim follows from this by iteration.
\end{proof}

Finally, we summarize some basic properties of the localized entropies $\cW_{x}(s)$ and $\cN_x(s)$ that we introduced in Definitions \ref{d:pointed_entropy} and \ref{d:nash_entropy}. The first couple of facts are clear from Theorem \ref{t:pfirstvar}.

\begin{proposition}\label{p:ptd_entropy}
The following hold for all $x\in M$ and all $s \in [-T,0)$.
\begin{enumerate}
 \item $\lim_{s\to 0}\cW_x(s)=0$.
 \item $\mu(g(-T),T) \leq \cW_x(s)\leq 0$.
 \item $\cW_x(s) = -\int_s^0 2|r| \int | \Ric + \nabla^2 f_{x} -\frac{g}{2|r|}|^2 \, d\nu_{x}(r)\,dr$.
\end{enumerate}
\end{proposition}

The following straightforward computation explains the use of the Nash entropy.

\begin{lemma}\label{l:basic_nash}
Let $u$ be a smooth positive solution to the conjugate
heat equation, of rapid decay and of unit mass.
Fix $t_0 \in \R$, put $\tau \equiv t_0  - t$ for $t < t_0$, and write $u \equiv (4\pi \tau)^{-\frac{n}{2}}e^{-f}$.
Then
\begin{align}
\frac{d}{dt}\left(\tau\int u \log u \, d\vol\right) = \mathcal{W}(g,f,\tau) + n + \frac{n}{2}\log 4\pi\tau.
\end{align}
\end{lemma}

Let us then summarize what we learn from this together with Proposition \ref{p:ptd_entropy}.

\begin{proposition}\label{p:nash_entropy}
The following hold for all $x \in M$ and $s \in [-T,0)$.
\begin{enumerate}
\item $\cW_x(s) \leq \cN_x(s) \leq 0$.
\item $\frac{d}{ds}\cN_x(s) = \frac{1}{|s|}({\mathcal{N}_x(s) - \cW_x(s)}) \geq 0$.
\item \label{e:nash_property}$\cN_x(s) = -\int \log H_x(s) \,d\nu_x(s) - \frac{n}{2}(1 + \log 4\pi|s|) = \int f_{x}(s)\,d\nu_{x}(s) - \frac{n}{2}$.
 \item $\cN_x(s) = -\int_s^{0} 2|r|(1 - \frac{r}{s})\int |\Ric + \nabla^2f_x - \frac{g}{2|r|} |^2\,d\nu_x(r)\,  dr$.
\end{enumerate}
\end{proposition}

\subsection{Heat kernel estimates}\label{s:hkest}

Here we provide careful statements and applications of some useful estimates due to Zhang \cite{zhang_gradient, zhang1, zhang2}.  As usual, we let $(M^n, g(t))$ denote a Ricci flow as in (\ref{e:RF}).

In \cite{zhang1}, Davies's method \cite{D} for deriving $L^\infty$ heat kernel estimates from
log-Sobolev inequalities is applied to the Ricci flow, based on Corollary \ref{c:rflogsob}. During the proof, the scalar curvature term in (\ref{e:rflogsob}) gets compensated by the evolution of the Riemannian measure, so that the final result holds without any upper scalar curvature assumptions.

\begin{theorem}\label{t:upper_rf_heat_kernel}
Define two auxiliary quantities
\begin{equation}\label{e:upper_hk_aux}
\rho \equiv \|R[g(-T)]^-\|_\infty, \;\, \mu \equiv \inf_{\tau \in (0,2T)} \mu(g(-T), \tau).
\end{equation}
Let $u: M \times [t_1,t_2] \to \R^+$ with $[t_1,t_2] \subseteq [-T,0]$ be a smooth positive solution to
\begin{equation}
\frac{\partial u}{\partial t} = \Delta_{g(t)} u.
\end{equation}
For each $t \in [t_1,t_2]$ we then have
\begin{equation}\label{e:uniform_upper_bound}
\|u(t)\|_\infty \leq  (4\pi (t-t_1))^{-\frac{n}{2}} e^{\rho (t-t_1) -\mu} \|u(t_1)\|_1.
\end{equation}
\end{theorem}

Hamilton \cite{Ham} proved a Harnack inequality for positive solutions to the heat equation on manifolds with lower Ricci bounds that only involves the gradient in the space variables. In \cite{zhang_gradient} this idea was applied to the Ricci flow, using the Bochner formula of Lemma \ref{l:rf_bochner} as the key tool.

\begin{theorem}\label{t:zhang_grad_est}
Let $u: M \times [t_1,t_2] \to \R^+ $ with $[t_1,t_2] \subseteq [-T,0]$ be a smooth positive solution to
\begin{equation}
\frac{\partial u}{\partial t} = \Delta_{g(t)}u.
\end{equation}
Then we have a spatial Harnack estimate
\begin{equation}\label{e:zhang_grad_est}
\left|\nabla \sqrt{\log \frac{\sup u}{u}}\right| \leq \frac{1}{\sqrt{t-t_1}}.
\end{equation}
\end{theorem}

The following corollary was not stated in \cite{zhang_gradient} but will be useful for us in Section \ref{s:lipschitz_nash}.

\begin{corollary}\label{c:hk_grad_est}
For each $C > 0$ there exists a $C' = C'(n,C) > 0$ such that if we have
\begin{align}
R[g(-s)] \geq -\frac{C}{|s|},\;\,\inf_{\tau \in (0,2|s|)} \mu(g(s),\tau) \geq -C,
\end{align}
and if we write $H(x,0\,|\,y,s) = (4\pi|s|)^{-\frac{n}{2}}\exp(-f_{x}(y,s))$ as before, then
\begin{align}
 |\nabla_x f_{x}|^2 \leq \frac{C'}{|s|}(C' + f_{x}).
\end{align}
\end{corollary}

\begin{proof}
Fix $y,s$ and let $u(x,t) \equiv H(x,t\,|\,y,s)$, so that $u$ solves the heat equation  by Lemma \ref{l:hk_properties}. We can apply Theorem \ref{t:upper_rf_heat_kernel} with $T = |s|$ and $[t_1,t_2] = [s + \epsilon, 0]$ for $\epsilon \to 0$ to conclude that
$$\bar{u} \equiv \sup_{[\frac{s}{2}, 0] \times M} u \leq C'|s|^{-\frac{n}{2}}.$$
On the other hand, Theorem \ref{t:zhang_grad_est} with $[t_1,t_2] = [\frac{s}{2}, 0]$ yields $|\nabla_xf_x|^2 \leq \frac{C'}{|s|} \log \frac{\bar{u}}{u}$ at  $(x,0)$.
\end{proof}

Finally, in \cite{zhang2}, Perelman's Harnack inequality from \cite{P} is used to prove a Gaussian lower bound for $H$ in terms of distance in the final time slice, by bringing in Theorems \ref{t:zhang_grad_est} and \ref{t:upper_rf_heat_kernel}.

\begin{theorem}\label{t:hk_lower_bound}
Define $\rho,\mu$ as in \eqref{e:upper_hk_aux}
and write $\tau \equiv t - s$ for $s < t$ in $[-T,0]$. Then
\begin{equation}\label{e:hk_lower_bound}
H(x,t\,|\,y,s) \geq  (8\pi\tau)^{-\frac{n}{2}}\exp\left(-\frac{4}{\tau}d_{g(t)}(x,y)^2 - \frac{1}{\sqrt{\tau}}\int_0^\tau \sqrt{\sigma}R(y,t-\sigma)\,d\sigma - \rho\tau + \mu\right).
\end{equation}
\end{theorem}

Notice that (\ref{e:hk_lower_bound}) involves the $g(t)$-distance because Theorem \ref{t:zhang_grad_est} bounds the $x$-gradient.

\section{Log-Sobolev and Gaussian concentration}

Section \ref{s:proof_ineqs} proves Theorem \ref{t:ineqs} using methods in the spirit of the papers \cite{BL, B}, which deal with the analogous problem for static Riemannian manifolds with $\Ric \geq 0$. In fact, this proof is not far removed from the usual proof of a log-Sobolev under the Bakry-{\'E}mery condition (\ref{e:BE}).  Section \ref{s:concentration} deduces the Gaussian concentration (Theorem \ref{t:concentration}), using a standard argument from the theory of log-Sobolev inequalities. Corollary \ref{c:gaussian_integral_upper} is then deduced as a consequence.

\subsection{Proof of the Poincar{\'e} and log-Sobolev inequalities}\label{s:proof_ineqs}

The starting point is to rewrite our two inequalities in a more convenient way. Writing $d\nu = d\nu_{x_0}(s)$ and passing to square roots in the log-Sobolev, we need to prove that for all $u \in C^\infty_0(M)$, with $u \geq 0$ in the second case,
\begin{align}\label{e:poinc_rewrite}
\int u^2 \,d\nu - \left(\int u\,d\nu\right)^2&\leq 2|s|\int |\nabla u|^2\,d\nu,\\
\label{e:logsob_rewrite}\int u\log u \,d\nu - \left(\int u \,d\nu\right)\log\left(\int u \,d\nu\right) &\leq |s|\int \frac{|\nabla u|^2}{u}\,d\nu.
\end{align}
Note that (\ref{e:logsob_rewrite}) in fact implies (\ref{e:poinc_rewrite}) by linearizing around $u \equiv 1$, but information about the equality case is lost in this way. However, we will prove (\ref{e:poinc_rewrite}) and (\ref{e:logsob_rewrite}) completely in parallel, with almost identical discussions of the respective equality cases.

The key insight, not unlike \cite{BE}, is that the heat kernel provides a homotopy between
the two terms that are being subtracted on the left-hand sides of (\ref{e:poinc_rewrite}), (\ref{e:logsob_rewrite}). The proof then reduces to deriving
a gradient estimate for the forward heat equation via the Bochner formula (Lemma \ref{l:rf_bochner}).

 In \cite{B}, this is done (in the static case) by bounding
the Hodge heat kernel on $1$-forms in terms of
the heat kernel on scalars using a Kato type inequality; unfortunately, see \cite{HSU}, this method loses a
factor of $n = {\rm rank}\;T^*M$, which is not accounted for in \cite{B}. While the same idea works for the Ricci flow, we will therefore instead follow in spirit the approach of \cite{BL}, where a precise gradient estimate is  obtained (in the static case) by applying the heat kernel homotopy
principle \emph{once again}.

For $s \leq t$ in $[-T,0]$, we write $P_{st}u$ for the evolution of $u \in C^\infty_0(M)$ from time $s$ to time $t$ under the forward heat equation coupled to the Ricci flow. In other words, by Lemma \ref{l:hk_properties},
\begin{equation}\label{e:hkrep}
(P_{st}u)(x) = \int u(y) H(x,t\,|\,y, s) \, d\vol_{g(s)}(y).
\end{equation}
Given this, the following lemma records the key homotopy principle.

\begin{lemma}\label{l:homotopy}
{\rm (1)} For any family of smooth functions $U_t$ parametrized by $t \in [-T,0]$,
\begin{equation}\label{e:homotopy}
\frac{d}{dt} P_{t0}U_t = P_{t0}\Box_tU_t.
\end{equation}

{\rm (2)} Let $u \in C^\infty_0(M)$ and put $u_t = P_{st}u$, so that $\Box_t u_t = 0$. Fix
$\phi,\psi: \R \to \R$. Then we have
\begin{align}
\label{e:commutator1}U_t \equiv \phi(u_t) &\Longrightarrow \Box U = - \phi''(u)|\nabla u|^2,\\
\label{e:commutator2}U_t \equiv \psi(u_t)|\nabla u_t|_{g(t)}^2 &\Longrightarrow \Box U = -2\psi(u)|\nabla^2u|^2 - 4\psi'(u)\langle \nabla^2u, du \otimes du\rangle - \psi''(u)|d u \otimes d u|^2,
\end{align}
omitting all subscripts $t$ and $g(t)$ on the right-hand side.
\end{lemma}

\begin{proof}
(1) The representation formula (\ref{e:hkrep}) yields
\begin{align*}
 \frac{d}{dt} P_{t0}U_t(x_0) = \frac{d}{dt} \int U_t(x) H(x_0,0\,|\,x, t) \, d\vol_{g(t)}(x).
\end{align*}
Now use that $H(x_0, 0\,|\,x, t)$ solves the conjugate heat equation in $(x,t)$ and $\frac{\partial}{\partial t} d\vol = -R\,d\vol$.

(2) The first claim is straightforward. The key to the second claim is Lemma \ref{l:rf_bochner}.
\end{proof}

Combining (\ref{e:homotopy}), (\ref{e:commutator1}), we can now rewrite the left-hand sides of (\ref{e:poinc_rewrite}), (\ref{e:logsob_rewrite}) as follows:
\begin{align}\label{e:heat_homotopy}
\int \phi(u) \,d\nu - \phi\left(\int u\,d\nu\right) &= -\int_s^0 \frac{d}{dt} P_{t0}(\phi(P_{st}u))(x_0) \,dt \notag\\
&=  \int_s^0 P_{t0}(\phi''(P_{st}u) |\nabla P_{st}u|^2_{g(t)})(x_0) \, dt,
\end{align}
with $\phi(x) = x^2$ and $\phi(x) = x\log x$, respectively.
We next estimate the integrand by combining (\ref{e:homotopy}),  (\ref{e:commutator2}) with $\psi = \phi''$, replacing $t,0$ by $r,t$ for a new variable $r \in [s,t]$. Thus, in the $x^2$ case,
\begin{align}
|\nabla P_{st}u|^2_{g(t)} = P_{st}(|\nabla u|^2_{g(s)}) - 2\int_s^t P_{rt}|\nabla^2 P_{sr}u |_{g(r)}^2\,dr.
\end{align}
Substituting this into (\ref{e:heat_homotopy}) immediately yields the Poincar{\'e} inequality of Theorem \ref{t:ineqs}(1).
In the $x\log x$ case, (\ref{e:commutator2}) simplifies quite drastically and we obtain
\begin{align}
\frac{|\nabla P_{st}u|^2_{g(t)}}{P_{st}u} = P_{st}\left(\frac{|\nabla u|^2_{g(s)}}{u}\right) - 2\int_s^t P_{rt}((P_{sr}u) |\nabla^2 \log P_{sr}u |_{g(r)}^2) \, dr\, .
\end{align}
Again we substitute this into (\ref{e:heat_homotopy}) to prove the required log-Sobolev inequality.  We can then finish the proof of Theorem \ref{t:ineqs} by observing that equality can occur if and only if either $u = const$, or the flow lines of $\nabla u$ or $\nabla \log u$ split off as isometric $\R$-factors.

\subsection{Gaussian concentration}\label{s:concentration}

Given the inequalities in Theorem \ref{t:ineqs}, proving the concentration estimate claimed in Theorem \ref{t:concentration} is by now a standard exercise in abstract metric measure theory: log-Sobolev inequalities yield Gaussian concentration (Herbst), and  Poincar{\'e} inequalities, which are weaker, still yield exponential concentration (Gromov-Milman). We refer to Ledoux \cite{L} for a good exposition; the following proof merely recalls the relevant points from \cite{L}.

\begin{proof}[Proof of Theorem \ref{t:concentration}]
We proceed from the log-Sobolev in the version (\ref{e:logsob}). The first idea is that this allows us to bound, in a specific fashion, the Laplace transform of $1$-Lipschitz functions with zero average. The second idea is to apply this bound to the distance function from a set.

As for the Laplace transform bound, fix $F \in C^\infty(M)$ with
$$
\int F \,d\nu = 0,\;\,|\nabla F| \leq 1.
$$
Let us define a Laplace type transform, or moment generating function, by
$$
U(\lambda) \equiv \frac{1}{\lambda} \log \int e^{\lambda F}\,d\nu.
$$
This is $O(\lambda)$ as $\lambda \to 0$. Moreover, (\ref{e:logsob}) applied to $\phi^2 = e^{\lambda F}/\int e^{\lambda F}\,d\nu$ easily yields
$$
\frac{dU}{d\lambda} \leq |s|
$$
for all $\lambda > 0$. Thus, altogether,
$$
\int e^{\lambda F}\,d\nu \leq e^{|s|\lambda^2}.
$$

We now apply the preceding inequality to the two test functions
$$
F \equiv \pm(G - \int G\,d\nu),\;\, G(y) \equiv {\rm dist}(y,B).
$$
Then we immediately obtain that
$$
e^{\lambda {\rm dist}(A,B)}\nu(A)\nu(B) \leq \int_A\int_B e^{\lambda(F(y_1) - F(y_2))} \, d\nu(y_1)\,d\nu(y_2) \leq e^{2|s|\lambda^2},
$$
and the theorem follows from this by optimizing in $\lambda$.
\end{proof}

\begin{proof}[Proof of Corollary \ref{c:gaussian_integral_upper}]
We begin by applying Theorem \ref{t:concentration} with $x_0$ replaced by $x_2$ and
$$
A = B_{r}(x_1, 0),\;\,B = B_{r}(x_2, 0).
$$
In order to derive (\ref{e:uppergauss2}) we must bound the factor $\nu(B)$ on the left-hand side of (\ref{e:concentration}) from below.
This can be done using two ingredients. First, using Zhang's Theorem \ref{t:hk_lower_bound},
$$\inf_{B_r(x_2,0)} H_{x_2}(s) \geq \frac{1}{C'}|s|^{-\frac{n}{2}}.$$
Second, the evolution of the volume form under Ricci flow and Theorem \ref{t:stronger_nlc} tell us that
$$\Vol_{g(s)}(B_r(x_2,0)) \geq \frac{1}{C'}\Vol_{g(0)}(B_r(x_2,0)) \geq \frac{1}{C'}r^n.$$
Together these show that $\nu(B) \geq \frac{1}{C'}$. Then (\ref{e:uppergauss2}) follows by dividing through by $\Vol_{g(s)}(B_r(x_1,0))$, which we can bound from below by the same argument we just used for
$\Vol_{g(s)}(B_r(x_2,0))$.

Finally,  Theorem \ref{t:hk_lower_bound}  also tells us that
$$\inf_{B_r(x_1,0)} H_{x_2}(s) \geq \frac{1}{C'}|s|^{-\frac{n}{2}}\exp\left(-\frac{C'}{|s|}d_{g(0)}(x_1,x_2)^2\right),$$
which allows us to deduce (\ref{e:distance_distortion}) from (\ref{e:uppergauss2}).
\end{proof}

\section{Lipschitz continuity of the pointed Nash entropy}\label{s:lipschitz_nash}

We now prove Theorem \ref{t:ent_cont} as a consequence of Corollary \ref{c:hk_grad_est} and our weighted Poincar{\'e} (\ref{e:poinc}). The standing assumption throughout this section is that $(M^n, g(t))$ is a Ricci flow parametrized by $t \in [-T,0]$ satisfying (\ref{e:gen_assns_2}). In order to prove that the mapping $x \mapsto f_x(s)H_x(s)$ from $(M, g(0))$ to $L^1(M, d\vol_{g(s)})$ is Lipschitz, clearly all we need to do is estimate the integral
\begin{align}
\cI \equiv \int |\nabla_x (f_x(y,s) H_x(y,s))| \,d\vol_{g(s)}(y).
\end{align}
Inserting the expression for $H_x$ in terms of $f_x$, and writing $d\nu \equiv d\nu_x(s)$, we find that
\begin{align}\label{e:n1}
\cI = \int |\nabla_xf_x-f_x\nabla_x f_x| \, d\nu \leq \|\nabla_x f_x\|_2 (1 + \|f_x\|_2),
\end{align}
where the subscript $2$ indicates the $L^2$-norm with respect to the probability measure $\nu$.
Now using Corollary \ref{c:hk_grad_est} we have the pointwise estimate
\begin{align}
|\nabla_x f_x|^2\leq \frac{C'}{|s|}(C'+f_x).
\end{align}
If we substitute this into (\ref{e:n1}) we therefore get
\begin{align}\label{e:n2}
\cI \leq C'|s|^{-\frac{1}{2}}\left(1+\int |f_x|^2\,d\nu\right).
\end{align}

To deal with this term we need the Poincar{\'e} inequality (\ref{e:poinc}) and some simple observations based on monotonicity of the Nash entropy. Precisely, the following lets us complete the proof.

\begin{theorem}\label{t:hk_upper_grad_int}
Under \eqref{e:gen_assns_2} the following hold.
\begin{enumerate}
\item  $\int f_x \, d\nu \in [\frac{n}{2} -C, \frac{n}{2}]$.
\item  $\int |\nabla f_x|^2\, d\nu\leq (\frac{n}{2} + C)\frac{1}{|s|}$.
\item  $\int |f_x|^2 \, d\nu \leq 2(n + 2C)^2$.
\end{enumerate}
\end{theorem}
Notice carefully that $\nabla f_x$ means $\nabla_y f_x$ here, not $\nabla_x f_x$ as in (\ref{e:n1}).
\begin{proof}[Proof of Theorem \ref{t:hk_upper_grad_int}]
The first statement follows from Propositions \ref{p:ptd_entropy} and \ref{p:nash_entropy}. For the second statement, notice that the inequality
$\cW_x(s)\leq \cN_x(s)$ is equivalent to
$$
\int (|\nabla f_x|^2+R)\, d\nu \leq \frac{n}{2|s|}\, .
$$
Finally, the Poincar\'e inequality of Theorem \ref{t:ineqs} gives us
\begin{align}\label{e:n3}
\int |f_x|^2 \, d\nu \leq 2|s|\int |\nabla f_x|^2 \, d\nu +\left(\int f_x\, d\nu\right)^2,
\end{align}
and the two terms on the right-hand side are bounded by the first two statements.
\end{proof}

\section{Proof of the $\epsilon$-regularity theorem}\label{s:epsilon_regularity}

Throughout this section we are considering a Ricci flow $(M^n,g(t))$ as in (\ref{e:RF}) that satisfies (\ref{e:basic_control}). We begin by proving Proposition \ref{p:eps_regularity_entropy}, which is a more restrictive version of Theorem \ref{t:eps_regularity}, and then use the continuity statement of Theorem \ref{t:ent_cont} to complete the proof of Theorem \ref{t:eps_regularity} in full.

For a point $y\in M$ on the $0$-time slice we define the normalized time-scale
\begin{align}\label{e:norm_time_scale}
t(y) \equiv {-\min}\{T, r_{|\Rm|}(y,0)^2\}.
\end{align}

\begin{proposition}\label{p:eps_regularity_entropy}
There exists an $\epsilon = \epsilon(n,C)>0$ such that if we have
\begin{align}\label{e:eps_assumption_entropy}
\forall y \in B_\delta(x,0): \mathcal{N}_{t(y)}(y) \geq -\epsilon
\end{align}
for some $x \in M$ and $0<\delta \leq \sqrt{T}$, then we also have
\begin{align}
\forall y\in B_\delta(x,0): r_{|\Rm|}(y,0) \geq \epsilon \cdot d_{g(0)}(y,\partial B_\delta(x,0)).
\end{align}
\end{proposition}

\begin{proof}
(1) By rescaling there is no harm in assuming $\delta = 1 \leq T$.
Let us assume for some $n,C>0$ that the lemma fails.  In this case we have for all $i\in\dN$ that there exists a complete Ricci flow
$$
(M^n_i,g_i(t),(x_i,0))
$$
with $t\in[-1,0]$ and $x_i\in M_i$ such that for each $y \in B_1(x_i,0)$ we have
\begin{align}\label{e:eps0}
\mathcal{N}_{t(y)}(y) \geq -\frac{1}{i},
\end{align}
whereas any point $y_i\in B_1(x_i,0)$ minimizing the quantity
\begin{align}
w(y)\equiv \frac{r_{|\Rm|}(y,0)}{d_{g_i(0)}(y,\partial B_1(x_i,0))}
\end{align}
must necessarily satisfy
\begin{align}
0<w(y_i)\leq\frac{1}{i}.
\end{align}

(2) Choose any such $y_i$, define $r_i \equiv r_{|\Rm|}(y_i,0)$, and consider the rescaled Ricci flows
$$
(\tilde{M}_i,\tilde g_i(t), (\tilde y_i,0)), \;\, \tilde g_i(t) \equiv \frac{1}{r_i^2}g_i(r_i^2 t), \;\,
t \in [-\frac{1}{r_i^2},0].$$
Certainly $r_{|\Rm|}(\tilde y_i,0)=1$, and $\tilde y_i$ moves away from the boundary in that
\begin{align}\label{e:eps2}
d_i \equiv \frac{1}{2}d_{\tilde g_i(0)}(\tilde y_i,\partial B_{\frac{1}{r_i}}(\tilde x_i,0)) \geq \frac{i}{2}.
\end{align}
On the other hand, since $y_i$ minimizes $w$, we find that
\begin{align}\label{e:eps3}
\tilde y \in B_{d_i}(\tilde y_i, 0) \;\Longrightarrow\; r_{|\Rm|}(\tilde y,0)\geq \frac{1}{2}.
\end{align}
Moreover, by Perelman's no local collapsing, see Theorem \ref{t:stronger_nlc}, we have that
\begin{align}
\Vol_{\tilde g_i(0)}(B_1(\tilde y,0)) \geq\kappa(n).
\end{align}
Thus we have uniform smooth bounds on $P_{1/4}(\tilde y,0)$ for all $\tilde y \in B_{d_i}(\tilde y_i,0)$,
and uniform noncollapsing at scale $1$ on $B_{d_i}(\tilde y_i,0)$. This allows us to pass to a subsequence to derive a pointed $C^\infty$ limit
$$
(\tilde M_i,\tilde g_i(t),(\tilde y_i,0))\rightarrow (\tilde M_\infty, \tilde g_\infty(t),(\tilde y_\infty,0)).
$$
The limit exists for $t \in [-\frac{1}{16}, 0]$ and is complete with bounded curvature while satisfying
\begin{align}\label{e:eps4}
r_{|\Rm|}(\tilde y_\infty,0)=1.
\end{align}

(3) To contradict this with (\ref{e:eps0}) we need a few estimates for the heat kernel. Namely,
\begin{equation}\label{e:aux_hk_bounds}
(4\pi|t|)^{-\frac{n}{2}} \exp(-l^{|t|}_{\tilde y_i}(\tilde y))\leq H(\tilde y_i,0\,|\,\tilde y,t) \leq C'(n,C)|t|^{-\frac{n}{2}}
\end{equation}
for all $t \in [-1,0)$ and $\tilde y \in \tilde M_i$ in the rescaled flows $(\tilde M_i,\tilde g_i(t), (\tilde y_i,0))$, where $l$ denotes the reduced length function of Perelman. The lower bound follows from Perelman's Harnack inequality \cite{P}, and the upper bound follows from Zhang's Theorem \ref{t:upper_rf_heat_kernel}, originally proved in \cite{zhang1}.

To clarify this point, note that in principle the required heat kernel estimates already follow from (\ref{e:eps3}) on $B_{d_i}(\tilde y_i,0)$
and neither depend on any special structure of the Ricci flow nor on the assumed value of $C$. We used some results from \cite{P, zhang1} instead just for the sake of simplicity.

(4) As usual, let us write
\begin{align}
H(\tilde y_i, 0 \,|\, \tilde y,t) \equiv (4\pi |t|)^{-\frac{n}{2}}e^{-f_i(\tilde{y},t)}.
\end{align}
Then the two bounds in (\ref{e:aux_hk_bounds}) together with the local regularity (\ref{e:eps3}) tell us that
\begin{align}\label{e:eps5}
f_i(\tilde y,t) \to f_\infty(\tilde y,t)
\end{align}
smoothly on compact subsets because the heat kernels $H_{\tilde y_i}$ satisfy uniform derivative bounds. Now the entropy smallness (\ref{e:eps0}) together with Proposition \ref{p:nash_entropy} shows that
\begin{align}
\int_{-\frac{1}{16}}^{0} 2|t|(1 - 16|t|)\int_{\tilde{M}_i} \left|\Ric[\tilde{g}_i] + \nabla^2f_i - \frac{\tilde{g}_i}{2|t|} \right|^2 d\nu_{\tilde{y}_i}(t)\, dt \leq \frac{1}{i}.
\end{align}
Thus, by Fatou's lemma, the function $f_\infty$ constructed in (\ref{e:eps5}) is a soliton potential for the limiting Ricci flow with singular time $t = 0$. Hence the $t$-time slice of the limiting flow is isometric to the $16|t|$-rescaling of the $(-\frac{1}{16})$-time slice. But then the only way for the curvature to stay bounded as $t \to 0$ is if $(\tilde{M}_\infty,\tilde g_\infty(t))$ is flat for all $t$. This contradicts (\ref{e:eps4}) and thus proves the lemma.
\end{proof}

We are now in good shape to in fact prove a strengthening of our main theorem.

\begin{theorem}\label{t:eps_regularity_2}
Theorem \ref{t:eps_regularity} is true and we can even replace \eqref{e:entropy_small} by $\cN_{x_0}(s) \geq -\epsilon$.
\end{theorem}
\begin{proof}
Define $\delta \equiv \min\{1,\frac{1}{2C'_{\ref{t:ent_cont}}}\epsilon_{\ref{p:eps_regularity_entropy}}\}$ and $\epsilon \equiv \frac{\delta}{2}\epsilon_{\ref{p:eps_regularity_entropy}}$. We can assume $-s = T = 1 \geq \delta$. Then
$$
\forall x\in B_\delta(x_0,0): \cN_x(1)\geq -\epsilon_{\ref{p:eps_regularity_entropy}}
$$
by Theorem \ref{t:ent_cont}. Thus Proposition \ref{p:eps_regularity_entropy} tells us that $r_{|\Rm|}(x_0,0) \geq \epsilon_{\ref{p:eps_regularity_entropy}}\delta \geq \epsilon$.
\end{proof}

\end{document}